\documentclass{dcds-b}
\usepackage{amsmath,amssymb,amsfonts}
\usepackage{paralist}
\usepackage{mathrsfs}
\usepackage{breakurl}

\usepackage[colorlinks=true]{hyperref}
\hypersetup{urlcolor=blue,citecolor=red}

\def\currentvolume{XX}
 \def\currentissue{X}
  \def\currentyear{201X}
   \def\currentmonth{XX}
    \def\ppages{1--XX}
     \def\DOI{10.3934/dcdsb.201x.14.xx}

\newtheorem{theorem}{Theorem}[section]
\newtheorem{corollary}{Corollary}

\newtheorem{lemma}[theorem]{Lemma}

\newtheorem{assumption}{Assumption}

\theoremstyle{definition}
\newtheorem{definition}[theorem]{Definition}
\newtheorem{remark}{Remark}

\newcommand{\bp}{\boldsymbol{p}}
\newcommand{\bx}{\boldsymbol{x}}

\newcommand{\dist}{\mathop{\mathrm{dist}}}
\newcommand{\diam}{\mathop{\mathrm{diam}}}
\newcommand{\coneN}{\mathbb{K}^{N}_{+}}
\newcommand{\intconeN}{\overset{\circ}{\mathbb{K}}{\vphantom{\mathbb{K}}}^{N}_{+}}
\newcommand{\simpN}{\Sigma_{N}}
\newcommand{\intsimpN}{\overset{\circ}{\Sigma}_{N}}

\title[Stability of random Markov chains]
{Asymptotic behaviour of  random tridiagonal Markov chains in biological applications}%
\author[Peter E. Kloeden and Victor S. Kozyakin]{}


\email{kloeden@math.uni-frankfurt.de}
\email{kozyakin@iitp.ru}

\thanks{The first author was DFG grants KL~1203/7-1, the   Spanish
Ministerio de Ciencia e Innovaci\'{o}n project
MTM2011-22411, the Consejer\'{\i}a de Innovaci\'{o}n, Ciencia y Empresa
(Junta de Andaluc\'{\i}a) under the Ayuda 2009/FQM314 and
the Proyecto de Excelencia P07-FQM-02468.  The second author was
partially supported by the
Russian Foundation for Basic Research, project no.
10-01-93112.}

\begin{document}

\maketitle \centerline{\scshape Peter E. Kloeden}
\medskip
{\footnotesize
 \centerline{Institut f\"ur  Mathematik, Goethe Universit\"at}
  \centerline{D-60054 Frankfurt am Main, Germany}
} 
\medskip

\centerline{\scshape Victor S. Kozyakin}
\medskip
{\footnotesize
 \centerline{Institute for Information Transmission Problems}
 \centerline{Russian Academy of Sciences}
  \centerline{Bolshoj Karetny lane 19, Moscow 127994 GSP-4, Russia}
} 

\bigskip

\centerline{Dedicated to the memory of Alexei Pokrovsky}

\begin{abstract}
Discrete-time discrete-state random Markov chains with a tridiagonal
generator are shown to have a random attractor consisting of singleton
subsets, essentially a random path, in the simplex of probability vectors.
The proof uses the Hilbert projection metric and the fact that the linear
cocycle generated by the Markov chain is a uniformly contractive mapping
of the positive cone into itself. The proof does not involve probabilistic
properties of the sample path $\omega$ and is thus equally valid in the
nonautonomous deterministic context of Markov chains with, say,
periodically varying transitions probabilities, in which case the
attractor is a periodic path.
\end{abstract}

\section{Introduction}\label{intro}

Markov chains with a tridiagonal generator are common in biological
applications, see, e.g., \cite{AllenL:10,AEH:09,WodKom:05}. Their asymptotic
behaviour is well understood when the transitions probabilities are
constant, i.e, the Markov chain is homogeneous or, equivalently, autonomous
in the language of dynamical systems. In this paper we consider the case
where the transitions probabilities can vary in time, e.g., periodically, or
randomly due to a periodically or randomly changing environment. The Markov
chains are then nonautonomous or random dynamical systems
\cite{ArnoldL:98,Chueshov:02,PKMR} and the concepts of autonomous dynamical
systems such as equilibria are inadequate. Thus a new concept of
nonautonomous or random attractors is needed.

The results in this paper are presented in the context of random Markov
chains and random dynamical systems, although the proofs are do not depend
at all on probabilistic properties of the sample path parameter $\omega$ and
are thus equally valid in the nonautonomous deterministic context of Markov
chains with, say, periodically varying transitions probabilities.

Tridiagonal Markov chains, both deterministic and random, are presented in
Section~\ref{S-model}. The long term dynamical behaviour of the autonomous
deterministic case is then given for completeness in
Section~\ref{D-dynamics}, although it follows as a special case, because it
provides useful background information for the random case. The proof uses
the fact that random Markov chains generate contractive linear cocycles
which map a positive cone into itself. First, in Section~\ref{S-CoMetric}, a
general theorem on existence of a random attractor with singleton subsets in
a metric space is formulated and proved. The assumptions seem rather
restrictive at the first sight, but are just what is needed later. Then, in
Section~\ref{S-HilbBirk} some preliminaries from the theory of positive
linear operators are recalled, the essence of which is that in a quite a
general situation, for linear maps positive with respect to the same
invariant cone, there exists a \emph{common} metric, the Hilbert projective
metric, in which all these linear mappings are \emph{uniformly contractive}.
Finally, in Section~\ref{S-LinCocycle}, it is shown that the linear cocycles
generated by the random Markov chains in Section~\ref{S-model} satisfy the
conditions of the abstract theorem from Section~\ref{S-CoMetric} under
uniform upper and lower positivity bounds on the tridiagonal transitions
probabilities and thus have a random attractor consisting of singleton
subsets. The random attractor is essentially a randomly varying path in the
simplex of probability vectors which pathwise attracts all other iterates of
the Markov chain. In the nonautonomous deterministic setting with periodical
transitions probabilities it is a periodic path.

There is an extensive literature on products of random Markov chains, see
e.g.,
\cite{Cohn:IJMMS,Hartfiel:98,Hartfiel:02,Leiz:LAA92,NSch:LAA99,Wolf:PAMS63}.
Although the problem investigated here does not seem to have been addressed
directly as such yet in the literature, the results could probably be
obtained by extending the proofs in \cite{Hartfiel:98,Wolf:PAMS63} after
similar computations to those that are needed below. The proof given here is
preferable since it is direct and is written in the language of random
dynamical systems. In particular, the paper demonstrates that the effect of
a random attractor to be singleton is valid not only for some classes of
monotone random systems, as it is reported earlier in \cite{ChuSche:04}, but
also for the class of random Markov chains studied below.

\section{Tridiagonal Markov chains in biological models}\label{S-model}

Markov chains with tridiagonal transition matrices are common in biological
models, for example, birth-and-death processes \cite{AllenL:10}, cell-cell
communication \cite{AEH:09} and cancer dynamics \cite{WodKom:05}, to name
just a few.

To fix ideas, consider the distance $d(t_n)$ between two cells at time $t_n
= n \Delta$, which is supposed to take discrete values in $\{1,\ldots, N\}$,
essentially the distance that they can move in one unit of time, where
$d(t_n)$ can stay unchanged or change to $d(t_n)\pm 1$ with certain
probabilities. This can be formulated as an $N$ state discrete-time Markov
chain with states $\{1,\ldots, N\}$ corresponding to the value of $d(t_n)$.

Let $\bp(t_n) = (p_1(t_n),\ldots,p_N(t_n))^T$ be the probability vector for
the state of the system at time $t_n$. The dynamics are described by the
system of difference equations
\begin{eqnarray*}
p_1(t_{n+1}) & = & \left[1- q_1\Delta\right] p_1(t_n) + q_2 \, p_2(t_n)\Delta
\\
p_j(t_{n+1}) & = & q_{2j-3}\, p_{j-1}(t_n)\Delta + \left[1- \left(q_{2j-2} +q_{2j-1}\right)\Delta \right] \, p_j(t_n) + q_{2j} \, p_{j+1}(t_n)\Delta,
\\
& & \hspace*{6.5cm} \quad j=2, \ldots, N-1,
 \\
p_N(t_{n+1}) & = & q_{2N-3} p_{N-1}(t_n)\Delta + \left[ 1- q_{2N-2}\Delta \right] p_N(t_n),
\end{eqnarray*}
where
\[
q_{j}>0,\quad j=1,\ldots,2N-2,
\]
and the $p_j$ satisfy the probability constraints
\[
\sum_{j=1}^N p_j = 1, \quad p_j \geq 0,\qquad j=1, \ldots, N
\]

This is a vector-valued difference equation
\[
\bp(t_{n+1}) = \left[I_N+ \Delta Q \right] \bp(t_{n})
\]
on the simplex $\Sigma_N$ in $\mathbb{R}^N$ defined by
\[
\Sigma_N = \bigg\{ \bp = (p_1,\cdots,p_N)^T :~ \sum_{j=1}^ N p_j = 1, ~ p_1,\ldots,p_N \in[0,1] \bigg\},
\]
where $I_N$ is the $N\times N$ identity matrix and $Q$ is the
 tridiagonal $N\times N$-matrix
\begin{equation}\label{E-defQ}
Q = \left[\begin{array}{ccccccc}
-q_1 & q_2 & & & & & \bigcirc\\
q_1 & -(q_2+q_3) & q_4 & & & & \\
 & \ddots & \ddots & \ddots & \ddots & \ddots & \\
 & & & & q_{2N-5} & -(q_{2N-4}+q_{2N-3}) & q_{2N-2} \\
\bigcirc & & & & & q_{2N-3} & - q_{2N-2}
\end{array}\right]
\end{equation}

This is a discrete-time finite-state Markov chain
\begin{equation}\label{DDE}
\bp^{(n+1)} = \left[I_N+ \Delta Q \right] \bp^{(n)}
\end{equation}
with the transition matrix $\left[I_N+ \Delta Q \right]$. It is a first
order linear difference equation on $\Sigma_N$ corresponding to the Euler
numerical scheme for the ordinary algebraic-differential equation
\[
\frac{d\bp}{dt} = Q \bp, \quad \bp \in
\Sigma_N,
\]
with the constant time step $\Delta > 0$.

\subsection{Random Markov chains}\label{S-random}

Let $(\Omega, \mathcal{F},\mathbb{P})$ be a probability space and suppose
now that the coefficients in the $Q$ matrix are random, i.e., the $q_{j} :
\Omega \to \mathbb{R}$ are $\mathcal{F}$-measurable mappings or,
equivalently, $Q : \Omega\to \mathbb{R}^{N\times N}$ is an
$\mathcal{F}$-measurable $N\times N$-matrix valued mapping.

This corresponds to a random environment, which is supposed to vary or be
driven by a stochastic process modelled by a metrical (i.e., measurable)
dynamical system $\Theta = \{\theta_n, n \in \mathbb{Z}\}$ on $\Omega$
generated by a bi-measurable invertible mapping $\theta : \Omega \to
\Omega$. In particular, $\Theta$ satisfies $\theta_{0}\omega = \omega$ and
\[
\theta_{m+n}\omega\equiv \theta_{m}(\theta_{n}\omega),\quad\forall~m,n\in\mathbb{Z},~\omega\in\Omega.
\]
See Arnold \cite{ArnoldL:98} for more information.

Define $L_{\omega} := I_N+ \Delta Q(\omega)$ and assume that $\Delta > 0$ is
sufficiently small, so that the eigenvalues of each matrix for given
$\omega$ lie in the unit disc of the complex plane (see next section).

This gives the \emph{random Markov chain}
\begin{equation} \label{RDE}
\bp^{(n+1)} = L_{\theta^n \omega} \bp^{(n)},
\end{equation}
which is a random linear difference equation on $\Sigma_N$, see
\cite{ArnoldL:98,Chueshov:02,IK:StochDyn03} for random difference equations.
The iterates of \eqref{RDE} are random probability vectors in $\Sigma_N$,
i.e., $\mathcal{F}$-measurable mappings $\bp : \Omega \to \Sigma_N$.

\section{Dynamical behaviour: deterministic case}\label{D-dynamics}

It is well-known that, under certain nondegeneracy conditions, the
``deterministic" Markov chain \eqref{DDE} has a unique equilibrium state
which is globally asymptotically stable in $\Sigma_N$. This result follows
as a special case of the main result of this paper below. A direct proof
using elementary methods will now be given, since it also provides useful
background information for the general ``random" case.

Let $\boldsymbol{1}_N$ be the column vector in $\mathbb{R}^N$ with all
components equal to $1$. Then
\begin{equation}\label{E-oneQ2zero}
\boldsymbol{1}_N^T Q = \boldsymbol{0},
\end{equation}
i.e. each column of $Q$ adds to zero. Moreover, $\boldsymbol{1}_N^T I_N  =
\boldsymbol{1}_N$, so $\boldsymbol{1}_N$ is a left eigenvector corresponding
to eigenvalue $\lambda = 1$ of the matrix $I_N+ \Delta Q $. Note that the
matrix $I_N+ \Delta Q$ is a stochastic matrix.

The Perron-Frobenius theorem applies to the matrix $L_{\Delta}:= I_N+ \Delta
Q$ when $\Delta > 0$ is chosen sufficiently small. In particular, it has
eigenvalue $\lambda = 1$ and there is a positive eigenvector $\bar{\bx}$
which can be normalized (in the $\|\cdot\|_1$ norm) to give a probability
vector $\bar{\bp}$, i.e. $[I_N+ \Delta Q]\bar{\bp} = \bar{\bp}$, so
$Q\bar{\bp} = 0$. In fact, we can show these properties directly for the
given matrix.

One can solve $Q\bar{\bx} = 0$ uniquely in $\mathbb{R}^N_{+}$ (up to a
scalar multiplier) since by assumption the $q_{j}>0$. Specifically,
\[
\bar{\bx}_{j+1} = \frac{q_{2j-1}}{q_{2j}}
\bar{\bx}_{j}, \quad j = 1, \ldots, N-1.
\]
Taking $\bar{\bx}_{1} = 1$ yields
\[
\bar{\bx}_{j+1} = \prod_{i=1}^j \frac{q_{2i-1}}{q_{2i}},
\quad j = 1, \ldots, N-1,
\]
and, hence, the probability vector
\[
\bar{\bp}_{1} = \frac{1}{\|\bar{\bx}\|_1}, \quad
\bar{\bp}_{j+1} = \frac{1}{\|\bar{\bx}\|_1}
\prod_{i=1}^j \frac{q_{2i-1}}{q_{2i}}, \qquad j = 1, \ldots, N-1,
\]
where
\[
\|\bar{\bx}\|_1 = \sum_{j=1}^{N}\bar{\bx}_{j} = 1+
\sum_{j=1}^{N-1}\prod_{i=1}^j \frac{q_{2i-1}}{q_{2i}}
\]
The corresponding Markov chain is, in fact, ergodic since by assumption all
the $q_j$ are positive. In particular, when $q_{2j-1} = q_{2j}$ for each
$j$, then $\bar{\bp}$ is the uniformly distributed probability vector with
identical components $\bar{\bp}_i = \frac{1}{N}$, $i=1,\ldots,N$.

\begin{theorem}
Let $q_{i}>0$ for $j=1,2,\ldots,2N-2$. Then the probability eigenvector
$\bar{\bp}$ is an asymptotically stable steady state of the difference
equation \eqref{DDE} on the simplex $\Sigma_N$.
\end{theorem}

\begin{proof}
First note that by Geshgorin's theorem applied to columns the eigenvalues of
the matrix $Q$ lie in the union of the closed discs centered on $-q_1 + 0
\imath, \ldots, -q_{2j-2}-q_{2j-1}+ 0 \imath, \ldots,  - q_{2N-2} +0 \imath$
in the complex plane with respective radii $q_1, \ldots, q_{2j-2}+q_{2j-1},
\ldots, q_{2N-2}$. These all contain the origin $0 + 0 \imath$ on their
boundary, but otherwise lie in the negative real part of the complex plan.
It is already known that $0$ is an eigenvalue, so all other eigenvalues have
strictly negative real parts. Moreover, $0$ is a simple eigenvalue with the
positive eigenvector $\bar{\bp}$.

It is easy to show that no generalized eigenvectors exist, since one would
satisfy the equation $Q\bar{\bx} = 0 \bar{\bx} +\bar{\bp}$, i.e.,
$Q\bar{\bx} = \bar{\bp}$, which is impossible since the sum of components on
the left hand side is equal to $0$, while the sum on the right side is equal
to $1$.

It follows that $\bar{\bp}$ is an eigenvector of the matrix $L =I_N+ \Delta
Q$ corresponding to the simple eigenvalue $\lambda= 1$. Moreover, if
$\Delta$ is small enough, then all other eigenvalues of the matrix $L$ lie
inside the unit disc in the complex plane, i.e., satisfy $|\lambda| < 1$. It
then follows that all solutions $\bp^{(n)}$ of the difference equation
\eqref{DDE} on $\Sigma_N$ converge to $\bar{\bp}$.

This can be shown by adapting the proof of Theorem 10.9 in
\cite{NobleDaniel:77}. Consider the Jordan canonical decomposition $L Q = Q
J$, where $Q$ is the matrix of eigenvectors and generalized eigenvectors of
the matrix $L$ with $\bar{\bp}$ as its first column corresponding to a $1
\times 1$ Jordan block $[1]$ and the other Jordan blocks corresponding to
the other eigenvalues with $|\lambda| < 1$. Then $J^k$ converges to an $N
\times N$ matrix $Z$ with $z_{1,1} = 1$ and all other components $z_{i,j} =
0$. This implies that
\[
L^k \bp^{(0)}= QJ^k Q^{-1} \bp^{(0)}
\to Q Z Q^{-1} \bp^{(0)}
\]
for any $\bp^{(0)} \in \Sigma_N$. Now $QZ
=[\bar{\bp}|\boldsymbol{0}|\ldots|\boldsymbol{0}]$, so $Q Z Q^{-1}
\bp^{(0)}$ is a scalar multiple of $\bar{\bp}$. But $L$ maps $\Sigma_N$ into
itself, so this scalar multiple of $\bar{\bp}$ is, in fact, $\bar{\bp}$
itself, i.e., $L^k \bp^{(0)} \to \bar{\bp}$ as $k \to \infty$ for all
$\bp^{(0)} \in \Sigma_N$.
\end{proof}

\section{Random attractors of uniformly contracting
cocycles}\label{S-CoMetric}

Let $M$ be a complete metric space equipped with the metric $\rho$.

\begin{definition}\label{D-cocycle}
A map $F : \mathbb{Z}^{+}\times\Omega\times M \to M$ is called a
(discrete-time) cocycle on $M$ with respect to the driving system $\Theta$
if it satisfies the \emph{initial condition}
\[
F(0,\omega,x)= x,\quad x\in M,~ \omega\in\Omega,
\]
and the \emph{cocycle property}
\begin{equation}\label{E-cocycle}
F(n+m,\omega,x) =
F(n,\theta_{m}\omega,F(m,\omega,x)),\quad
x\in M,~\omega\in\Omega,~ m,n\in\mathbb{Z}^{+}.
\end{equation}
The pair $(\Theta,F)$ is called a (discrete-time) random dynamical system in
\cite{ArnoldL:98}.
\end{definition}

Define $f(\omega,x) := F(1,\omega,x)$. Then, clearly, due to cocycle
property \eqref{E-cocycle} for any $n \in \mathbb{Z}^{+}$ the map
$F(n,\omega,x)$ can be expressed as a superposition of maps $f(\omega,x)$
for different $\omega$ and $x$:
\begin{equation}\label{E-Frep}
 F(n,\omega,x)=f(\theta_{n-1}\omega,F(n-1,\omega,x))=
 f(\theta_{n-1}\omega,\dots f(\theta_{1}\omega,f(\theta_{0}\omega,x))\dots).
\end{equation}
The map $f(\omega,x)$ is called the \emph{generator} of the cocycle
$F(n,\omega,x)$.

In what follow it will be supposed that the cocycle $F(n,\omega,x)$ is
continuous in $x$ for every $n \in \mathbb{Z}^{+}$ and $\omega\in \Omega$.
By \eqref{E-Frep} the cocycle $F(n,\omega,x)$ is continuous in $x$ if its
generator $f(\omega,x)$ is continuous in $x$ for every $\omega \in \Omega$.

An $\mathcal{F}$-measurable family $\mathscr{A} = \{A_{\omega},
\omega\in\Omega\}$ of nonempty compact subsets of $M$  is the family  of
image sets of an  $\mathcal{F}$-measurable set valued mapping $\omega$
$\mapsto$ $A_{\omega}$,  i.e.,  for which  the real valued mapping $\omega$
$\mapsto$ $\dist_{\rho}(x,A_{\omega})$ is $\mathcal{F}$-measurable for each
$x$ $\in$ $M$, see \cite[Theorem 8.1.4]{AuFrank:09}. It is called
\emph{$f$-invariant} if $f(\omega,A_{\omega}) = A_{\theta \omega}$, and
hence
\[
F(n,\omega, A_{\omega}) = A_{\theta_n \omega}, \quad \omega \in \Omega,~
n \in\mathbb{Z}^{+}.
\]

Recall that the \emph{Hausdorff separation}, or semi--metric,
$H_{\rho}^{*}(X,Y)$, of the nonempty compact subsets $X$ and $Y$ of $M$ is
defined by
\[
H_{\rho}^{*}(X,Y) := \max_{x\in X} \mbox{\rm dist}_{\rho}(x,Y),
\]
where $\dist_{\rho}(x,Y) := \min_{y\in Y}\rho(x,y)$, and the \emph{Hausdorff
metric} $H_{\rho}(X,Y)$ for the nonempty compact subsets $X$ and $Y$ of $M$
is
\[
H_{\rho}(X,Y) := \max\left\{H_{\rho}^{*}(X,Y),H_{\rho}^{*}(Y,X)\right\}.
\]
Finally, the \emph{diameter} of a subset $X$ of $M$ is defined by $\diam(X)
:= \sup_{x,y\in X}\rho(x,y)$.

\begin{definition}\label{pba}

An $\mathcal{F}$-measurable family $\mathscr{A} =\{A_{\omega},
\omega\in\Omega\}$ of nonempty compact subsets of $M$ is called a
\emph{random attractor} if it (i) is $F$--invariant, (ii) pullback attracts
nonempty bounded subsets of $M$, i.e.,
\[
H_{\rho}^{*}(F(n,\theta_{-n} \omega,D),A_{\omega}) \to 0 \quad \textrm{as}\quad n
\to \infty,
\]
for all $\omega \in \Omega$ and nonempty bounded subsets $D$ of $M$, and
(iii) is the minimal family (under inclusion) satisfying (i) and (ii).
\end{definition}

\begin{definition}\label{D-dissip}
A cocycle $F : \mathbb{Z}^{+}\times\Omega\times M \to M$ is called
\emph{uniformly dissipative} if there exist a number $N_{d}
\in\mathbb{Z}^{+}$ and a closed bounded set $M_{0} \subset M$ such that
\begin{equation}\label{E-diss}
F(N_{d},\omega,M)\subseteq M_{0},\quad \forall~ \omega\in\Omega.
\end{equation}
\end{definition}

\begin{definition}\label{D-cont}
A cocycle $F : \mathbb{Z}^{+}\times\Omega\times M \to M$ is called
\emph{uniformly contractive} if there exist a number $N_{c}
\in\mathbb{Z}^{+}$ and a number $\lambda \in [0,1)$ such that
\begin{equation}\label{E-contr}
\rho(F(N_{c},\omega,x),F(N_{c},\omega,y))\le\lambda \rho(x,y),\quad
\forall~ \omega\in\Omega, ~x,y\in M.
\end{equation}
\end{definition}

We can now formulate an abstract theorem on existence of a random attractor.

\begin{theorem}\label{T-metric}
Let $F : \mathbb{Z}^{+}\times\Omega\times M \to M$ be a uniformly
dissipative and uniformly contractive cocycle. Then it has a random
attractor $\mathscr{A} = \{A_{\omega}, \omega\in\Omega\}$. Moreover, the set
$A_{\omega}$ consists of a single point for each $\omega \in \Omega$.
\end{theorem}

\begin{proof}
Fix numbers $N_{d}$, $N_{c} \ge 1$ and $\lambda < 1$ and a closed bounded
set $M_{0}$ for which \eqref{E-diss} and \eqref{E-contr} hold.

First, it will be shown that for any bounded set $D \subseteq M$ and any
$\omega \in \Omega$ the sets $F(n,\theta_{-n} \omega,D)$ converge to some
single-point set $A_{\omega}$ as $n \to \infty$. To do this, in fact, an
even stronger statement will be proved: for any $\omega \in \Omega$ the sets
\[
M_{n}(\omega):=F(n,\theta_{-n}\omega,M)
\]
converge to some single-point set $A_{\omega}$ as $n \to \infty$.

Note that for each $\omega$ the sets $M_{n}(\omega)$ are closed as images of
the closed set $M_{0}$ under continuous maps $F(n,\theta_{-n}\omega,\cdot)$.
Moreover, for any $\omega \in \Omega$ the sequence of sets $M_{n}(\omega)$
is nested under inclusion:
\begin{equation}\label{E-monot}
 M_{n+1}(\omega)\subseteq M_{n}(\omega),\quad n\ge 1.
\end{equation}
Indeed, by \eqref{E-cocycle}
\[
M_{n+1}(\omega):= F(n+1,\theta_{-(n+1)}\omega,M)=
F(n,\theta_{1}(\theta_{-(n+1)}\omega),F(1,\theta_{-(n+1)}\omega,M)).
\]
Clearly $\theta_{1}(\theta_{-(n+1)}\omega) = \theta_{-n}\omega$ and
$F(1,\theta_{-(n+1)}\omega,M) \subseteq M$, so
\[
M_{n+1}(\omega) :=F(n+1,\theta_{-(n+1)}\omega,M)\subseteq
F(n,\theta_{-n}\omega,M)=:M_{n}(\omega),
\]
and \eqref{E-monot} holds.

Set
\[
 d_{n}=\sup_{\omega\in\Omega}\diam (M_{n}(\omega)).
\]
It will be shown that
\begin{equation}\label{E-mondiam}
 d_{n+1}\le d_{n}\le \diam(M_{0}),\quad n\ge N_{d}.
\end{equation}
The inequality $d_{n+1} \le d_{n}$ follows from \eqref{E-monot} provided
that the both numbers $d_{n+1}$ and $d_{n}$ are finite. Thus
\eqref{E-mondiam} follows if it can be shown that
\[
 d_{N_{d}}\le \diam(M_{0}).
\]
This last inequality readily follows from the inclusion
\[
M_{N_{d}}(\omega)=F(N_{d},\theta_{-N_{d}}\omega,M)\subseteq M_{0},
\]
which is a direct corollary of \eqref{E-diss}.

The inequalities \eqref{E-mondiam} are thus proved, but now they will be
 strengthened to
\begin{equation}\label{E-mondiamL}
 d_{n+N_{c}}\le \lambda d_{n},\quad n\ge N_{d}.
\end{equation}
To prove this inequality note that by \eqref{E-cocycle}
\begin{eqnarray*}
M_{N_{c}+n}(\omega) & = & F(N_{c}+n,\theta_{-(N_{c}+n)}\omega,M)
\\[1.5ex]
& = & F(N_{c},\theta_{n}(\theta_{-(N_{c}+n)}\omega),F(n,\theta_{-(N_{c}+n)}\omega,M)).
\end{eqnarray*}
Here $F(n,\theta_{-(N_{c}+n)}\omega,M)=M_{n}(\theta_{-N_{c}}\omega)$, so
\[M_{N_{c}+n}(\omega)=F(N_{c},\theta_{-N_{c}}\omega,M_{n}(\theta_{-N_{c}}\omega))
\]
and, by the uniform contractivity of the cocycle $F(n,\omega,x)$,
\[
\diam\left(M_{N_{c}+n}(\omega)\right)\le \lambda \diam\left(M_{n}(\theta_{-N_{c}}\omega)\right).
\]
Taking the supremum over all $\omega \in \Omega$ in the above inequality
then gives \eqref{E-mondiamL}.

To finalize the proof of the theorem it remains to note that, for any given
$\omega \in \Omega$, the sequence of closed sets $\{M_{n}(\omega)\}$ is
nested under inclusion and, by \eqref{E-mondiam} and \eqref{E-mondiamL}, the
diameters of the sets $M_{n}(\omega)$ tend to zero as $n \to \infty$. Then,
by the Cantor Intersection Theorem (or Property), see, e.g.,
\cite[Th.~13.65]{TBB:08}, the intersection
\[
A_{\omega}= \bigcap_{n\geq 1} M_{n}(\omega)
\]
is nonempty and consists of exactly one point.
\end{proof}

\begin{remark}\label{Rem0}
The proof of Theorem~\ref{T-metric} implies that the component sets
$A_{\omega}$ of the random attractor $\mathscr{A} = \{A_{\omega},
\omega\in\Omega\}$ satisfy the inclusion $A_{\omega}\subseteq M_{0}$ for all
$\omega\in\Omega$.
\end{remark}

\begin{remark}\label{Rem1}
To prove Theorem~\ref{T-metric} it would suffice to require that
\eqref{E-contr} holds only for $x$, $y \in M_{0}$, provided that the cocycle
$F$ is uniformly dissipative.
\end{remark}

\begin{remark}\label{Rem2}
It would be preferable to formulate properties of dissipativity and
contractivity for a cocycle not in terms of the map $F(n,x,y)$, but in terms
of its generator $f(\omega,x)$. As will be seen below, in general, this is
not possible in some interesting and natural applications, where the arising
cocycle is uniformly dissipative and contractive, whereas neither the
dissipativity for $N_{d} =1$ nor the contractivity for $N_{c} = 1$ holds.
\end{remark}

\begin{remark}\label{Rem3}
Theorem~\ref{T-metric} has been formulated under rather severe assumptions.
These can be essentially weakened, but serve perfectly well for the purposes
of this paper.
\end{remark}

\section{The Hilbert metric and Birkhoff's theorem}\label{S-HilbBirk}

To apply the results from Section~\ref{S-CoMetric} to linear cocycles
generated by tridiagonal Markov chains \eqref{RDE}, first recall some
auxiliary facts (see, e.g., \cite{BVW:10,Bush:ARMA73,KLS:PosLinSys:e})
following the work \cite{ADFP:93}.

Denote by $\coneN$ the cone\footnote{Recall, see, e.g.
\cite{KLS:PosLinSys:e}, that a set $K$ in a Banach space is called a
\emph{cone} if it is convex, closed and $tK \subseteq K$ for any real $t \ge
0$, and $K\cap -K = \{0\}$.} of elements $x =(x_{1},x_{2},\ldots,x_{N})^T
\in \mathbb{R}^{N}$ with nonnegative components and by $\intconeN$ the
interior of $\coneN$, which is clearly non-empty. Then the quantity
\[
\vartheta(x,y)=\inf \big\{t:~ tx-y\in\coneN \big\}
\]
is finite valued for any $x$, $y \in \intconeN$.

\begin{definition}\label{D-HilbMet}
The function
\[
\rho_{H}(x,y)= \left|\ln\left(\frac{\vartheta(x,y)}{\vartheta(y,x)}\right)\right|
\]
is called the \emph{Hilbert projective metric} (or, sometimes, the Birkhoff
metric \cite{KLS:PosLinSys:e}).
\end{definition}

\begin{remark}
Definition~\ref{D-HilbMet} is applicable to cones in a general Banach space.
For the cone $\coneN$ in the finite-dimensional space $\mathbb{R}^{N}$, it
can be shown to be equal to
\begin{equation}\label{E-HRN}
\rho_{H}(x,y)=\left|\ln\left(\frac{\max_{i}y_{i}/x_{i}}{\max_{i}x_{i}/y_{i}}\right)\right|
\end{equation}
or
\[
\rho_{H}(x,y)=\left|
\max_{i}\left\{\ln y_{i}-\ln x_{i}\right\} -
\max_{i}\left\{\ln x_{i}-\ln y_{i}\right\}\right|
\]
for vectors $x = (x_{1},x_{2},\ldots,x_{N})^T$ and $y
=(y_{1},y_{2},\ldots,y_{N})^T$ in $\coneN$.
 \end{remark}

Observe that $\rho_{H}(x,y)$ satisfies the triangle inequality
\[
\rho_{H}(x,z)\le\rho_{H}(x,y)+\rho_{H}(y,x),\quad x,y,z\in\intconeN,
\]
whereas the relation $\rho_{H}(x,y)=0$ with $x,y\in\intconeN$ does not
implies the equality $x=y$, but only the equality $x = ty$ for some $t > 0$.
Moreover,
\[
\rho_{H}(sx,ty)=\rho_{H}(x,y),\quad \forall~s,t>0,~x,y\in\intconeN.
\]
Thus, strictly speaking, $\rho_{H}(x,y)$ is not a metric on $\intconeN$, but
only a semi-metric. It becomes a metric, however, on a projective space. An
important way to make it a proper metric is covered by the following
theorem.

\begin{theorem}\label{T-Hilb}
Let $X \subseteq \mathbb{R}^{N}$ be a closed, bounded set such that
$0\not\in X$ and any ray $\{tx: t>0\}$ with $x \in \intconeN$ intersects $X$
in at most one point. Then $(X\cap\intconeN, \rho_{H})$ is a metric space.

Moreover, if any ray $\{tx: t>0\}$ with $x \in \intconeN$ intersects $X$ in
exactly one point, then the metric space $(X\cap\intconeN, \rho_{H})$ is
complete.
\end{theorem}

\begin{corollary}\label{C-simplex}
The interior $\intsimpN$ of the probability simplex $\simpN$ is the complete
metric space with the Hilbert projective metric.
\end{corollary}

\begin{remark}\label{R-unbound}
In general, the set $X\cap\intconeN$ in Theorem~\ref{T-Hilb} is unbounded in
metric $\rho_{H}$. Convergence in the metric space $(X\cap\intconeN,
\rho_{H})$ implies convergence with respect to any norm in $\mathbb{R}^{N}$.
\end{remark}

\begin{definition}\label{D-poslinop}
A linear operator (matrix) $L : \mathbb{R}^{N} \to\mathbb{R}^{N}$ is called
\emph{nonnegative} if $L : \coneN \to\coneN$, and \emph{positive} if $L :
\coneN\setminus\{0\} \to\intconeN$.\footnote{It is straightforward to show
that a matrix $L$ is nonnegative iff its components are nonnegative, and
positive iff its components are strictly positive.}
\end{definition}

\begin{definition}\label{D-posconsts}
If a linear operator (matrix) $L : \mathbb{R}^{N} \to\mathbb{R}^{N}$ is
positive, then its \emph{projective diameter} $\delta(L)$ is defined by
\[
\delta(L)=\sup\left\{\rho_{H}(Lx,Ly):~ x,y \in \intconeN\right\}
\]
and the \emph{contraction ratio} $\varkappa(L)$ of $L$ is defined by
\[
\varkappa(L)=\inf\left\{\lambda:~
\rho_{H}(Lx,Ly)\le\lambda\rho_{H}(x,y),~x,y\in\intconeN\right\}
\]
\end{definition}

\begin{theorem}[Birkhoff]\label{T-Birk}
If $L : \mathbb{R}^{N} \to \mathbb{R}^{N}$ is a positive linear operator
(matrix), then
\[
\varkappa(L)\le\tanh\left(\frac{1}{4}\delta(L)\right)<1.
\]
\end{theorem}

Given a positive linear operator (matrix) $L : \mathbb{R}^{N}\to
\mathbb{R}^{N}$ denote by $\Tilde{L}(\cdot)$ the (nonlinear) operator
defined by
\[
\Tilde{L}(x)=P(Lx),\quad x\in\simpN,
\]
where $P$ is the projection operator to the simplex $\simpN$ defined by
\[
P(x)=\frac{1}{x_{1}+x_{2}+\dots+x_{N}}x,\quad x\in\coneN.
\]

\begin{theorem}\label{T-contsimp}
$\Tilde{L}(\cdot)$ is a contracting mapping with the contraction constant
$\varkappa(L)$ on the metric space $(\intsimpN,\rho_{H})$.
\end{theorem}

\begin{remark}\label{R-Acontr}
If the simplex $\simpN$ is invariant for the positive linear operator $L$
then the assertion of Theorem~\ref{T-contsimp} can be simplified to:
\emph{if $L : \mathbb{R}^{N} \to \mathbb{R}^{N}$is a positive linear
operator (matrix) for which $L\simpN \subseteq\simpN$, then it is a
contracting mapping with the contraction constant $\varkappa(L)$ on the
metric space $(\intsimpN,\rho_{H})$.}
\end{remark}

\section{Attractors of linear
cocycles with a tridiagonal generator}\label{S-LinCocycle}

The result of Sections~\ref{S-CoMetric} and \ref{S-HilbBirk} will be applied
here to the linear system generated by the tridiagonal Markov chains that
were introduced in Section~\ref{S-model}.

Let $\mathscr{L}$ be a set of linear operators $L_{\omega} :\mathbb{R}^{N}
\to \mathbb{R}^{N}$ parametrized by the parameter $\omega$ taking values in
some set $\Omega$ and let $\{\theta_{n}, n\in\mathbb{Z}\}$ be a
(discrete-time) group of maps of $\Omega$ onto itself. The maps
$L_{\omega}x$ serve as the generator of a linear cocycle
\[
F_{\mathscr{L}}(n,\omega)x=L_{\theta_{n-1}\omega}\cdots L_{\theta_{1}\omega}L_{\theta_{0}\omega}x.
\]

In particular, consider the case where
\begin{equation}\label{E-Lomega}
 L_{\omega}:=I_N+ \Delta Q(\omega)
\end{equation}
with the tridiagonal matrices $Q(\omega)$ of the form \eqref{E-defQ} for
which the band entries $q_{i}$ depend on the parameter $\omega \in \Omega$,
i.e., $q_{i} = q_{i}(\omega)$ for $i = 1, 2, \ldots, 2N-2$.

In the sequel the following basic assumption will be used.
\begin{assumption}\label{ass}
There exist numbers $0 < \alpha \le \beta < \infty$ such that the uniform
estimates
\begin{equation}\label{E-mainA}
\alpha\le q_{i}(\omega)\le\beta,\quad \omega\in\Omega,~ i=1,2,\ldots,2N-2,
\end{equation}
hold.
\end{assumption}

\begin{lemma}\label{L-Lnonneg}
Let $\Delta < \frac{1}{2\beta}$. Then the entries of all of the matrices
$L_{\omega}$, $\omega\in\Omega$, are nonnegative, and thus
\[
L_{\omega}\mathbb{K}^{N}_{+}\subseteq\mathbb{K}^{N}_{+},\quad \omega\in\Omega.
\]
\end{lemma}

\begin{proof}
This follows directly from the representation $L_{\omega} : =I_N+ \Delta
Q(\omega)$ and the form \eqref{E-defQ} of the matrices $Q(\omega)$.
\end{proof}

If $\Delta < \frac{1}{2\beta}$, then $\gamma
:=\min\{\Delta\alpha,1-2\Delta\beta\} > 0$. Define the subset
$\Sigma_{N}(\gamma)$ of the simplex $\Sigma_{N}$ as
\[
\Sigma_{N}(\gamma) = \big\{x=(x_{1},x_{2},\ldots,x_{N}):~
\sum_{i=1}^{N}x_{i}=1,~x_{1},x_{2},\ldots,x_{N}\ge\gamma^{N-1}\big\},
\]
and the subcone $\coneN(\gamma)$ of the cone $\coneN$ as

\[
\coneN(\gamma)= \big\{tx: t>0,~ x\in \Sigma_{N}(\gamma)\big\}.
\]

Clearly
\[
\coneN(\gamma)\setminus\{0\}\subseteq \intconeN.
\]

\begin{lemma}\label{L-Lpos}
Let $\Delta < \frac{1}{2\beta}$ and $\gamma
=\min\{\Delta\alpha,1-2\Delta\beta\}$. Then,
\begin{equation}\label{E-Adiss}
F_{\mathscr{L}}(N-1,\omega)\Sigma_{N}\subseteq \Sigma_{N}(\gamma),\quad \omega\in\Omega,
\end{equation}
and, hence,
\[
F_{\mathscr{L}}(N-1,\omega)(\coneN\setminus\{0\})\subseteq \coneN(\gamma)\setminus\{0\}\subseteq \intconeN,\quad \omega\in\Omega.
\]
\end{lemma}

\begin{proof}
Fix an $\omega\in\Omega$. By induction, it can be shown for each
$n=1,2,\ldots,N-1$ that the matrix
\[
F_{\mathscr{L}}(n,\omega)=L_{\theta_{n-1}\omega}\cdots L_{\theta_{1}\omega}L_{\theta_{0}\omega}
\]
is $(2n+1)$-diagonal and that its components belonging to the main diagonal
and to the first $n$ sub and super diagonals are greater than or equal to
$\gamma^{n}$, while others vanish. Thus
\begin{equation}\label{E-Entpos}
\text{\em all components
of the matrix~} F_{\mathscr{L}}(N-1,\omega) \text{\em ~are positive and
exceed~} \gamma^{N-1}.
\end{equation}
By \eqref{E-oneQ2zero} and \eqref{E-Lomega}
\[
\boldsymbol{1}_N^T L_{\omega} = \boldsymbol{1}_N^T,\quad \omega\in\Omega,
\]
and \eqref{E-Entpos} implies that
\[
F_{\mathscr{L}}(N-1,\omega)\intsimpN\subseteq
\Sigma_{N}(\gamma)\subset\intsimpN,\quad\omega\in\Omega.
\]
The inclusion \eqref{E-Adiss} is thus established.
\end{proof}

\begin{theorem}\label{T-linear}
Let $F_{\mathscr{L}}(n,\omega)x$ be the linear cocycle
\[
F_{\mathscr{L}}(n,\omega)x=L_{\theta_{n-1}\omega}\cdots L_{\theta_{1}\omega}L_{\theta_{0}\omega}x.
\]
with matrices $L_{\omega} := I_N+ \Delta Q(\omega)$, where the tridiagonal
matrices $Q(\omega)$ are of the form \eqref{E-defQ} with the entries $q_{i}
= q_{i}(\omega)$ satisfying the uniform estimates \eqref{E-mainA} in
Assumptiom \ref{ass}. In addition, suppose that $\Delta < \frac{1}{2\beta}$.

Then, the set $\simpN$ is invariant under $F_{\mathscr{L}}(N-1,\omega)$,
i.e.,
\[
F_{\mathscr{L}}(N-1,\omega)\simpN\subseteq
\simpN,\quad\omega\in\Omega.
\]
Moreover, the restriction of $F_{\mathscr{L}}(n,\omega)x$ to the set
$\simpN$ is a uniformly dissipative and uniformly contractive cocycle (with
respect to the Hilbert metric), which has a random attractor $\mathscr{A} =
\{A_{\omega}, \omega\in\Omega\}$ such that each set $A_{\omega}$,
$\omega\in\Omega$, consists of a single point.
\end{theorem}

\begin{proof}
Under the assumptions, $\gamma :=\min\{\Delta\alpha,1-2\Delta\beta\} > 0$.
Define
\[
\delta=\sup\left\{\rho_{H}(x,y):~ x,y\in\Sigma_{N}(\gamma)\right\}.
\]
It follows by formula \eqref{E-HRN} for the Hilbert metric in $\coneN$ that
$\delta<\infty$, so
\begin{equation}\label{E-defkappa}
\varkappa\le\tanh\left(\frac{1}{4}\delta\right)<1.
\end{equation}
Hence, by Lemmata~\ref{L-Lnonneg} and \ref{L-Lpos}, for each
$\omega\in\Omega$ the matrix $F_{\mathscr{L}}(N-1,\omega)$ satisfies the
conditions of Theorem~\ref{T-Birk} and is thus uniformly contractive in the
metric space $(\intsimpN,\rho_{H})$ with the contraction constant
$\varkappa$. Moreover, by Lemma~\ref{L-Lpos}, each of the matrices
$F_{\mathscr{L}}(N-1,\omega)$, $\omega\in\Omega$, maps the set $\intsimpN$
to the set $\Sigma_{N}(\gamma) \subset \intsimpN$, which is bounded in the
Hilbert metric $\rho_{H}$.

Setting $M = \intsimpN$ and $\rho = \rho_{H}$, $M_{0} =\Sigma_{N}(\gamma)$
and $N_{d} = N_{c} = N-1$, it follows that all of the conditions of
Theorem~\ref{T-metric} are satisfied. Hence the restriction of the cocycle
$F_{\mathscr{L}}(n,\omega)x$ to the set $\intsimpN$ has a random attractor
$\mathscr{A} =\{A_{\omega}, \omega\in\Omega\}$ such that each set
$A_{\omega}$, $\omega\in\Omega$, consists of a single point.

To complete the proof it remains to note only that due to Lemma~\ref{L-Lpos}
all the matrices $F_{\mathscr{L}}(N-1,\omega)$, $\omega\in\Omega$, not only
map the set $\intsimpN$ to the set $\Sigma_{N}(\gamma) \subset \intsimpN$,
but also map the larger set $\simpN$ to the set $\Sigma_{N}(\gamma) \subset
\intsimpN$. Hence, the family of sets $\mathscr{A} = \{A_{\omega},
\omega\in\Omega\}$ is a random attractor for the cocycle
$F_{\mathscr{L}}(n,\omega)x$ considered not only on the set $\intsimpN$, but
also on its closure $\simpN$.
\end{proof}

\begin{remark}\label{Rem5}
The proof of Theorem~\ref{T-linear} implies that the component sets
$A_{\omega}$ of the random attractor $\mathscr{A} = \{A_{\omega},
\omega\in\Omega\}$ satisfy the inclusion $A_{\omega}\subseteq
\Sigma_{N}(\gamma)$ for every $\omega\in\Omega$.
\end{remark}

Henceforth write $A_{\omega} = \{a_{\omega}\}$ for the singleton component
subsets of the random attractor $\mathscr{A}$. Then the random attractor is
an entire random sequence $\{a_{\theta_n \omega}, n \in \mathbb{Z}\}$ in
$\Sigma_{N}(\gamma)\subset\intsimpN$, which attracts other iterates of the
random Markov chain in the pullback sense. Pullback convergence involves
starting at earlier initial times with a fixed end time, see
\cite{CKS:NDST02,PKMR}. It is, generally, not the same as forward
convergence in the sense usually understood in dynamical systems, but in
this case it is the same due to the uniform boundedness of the contractive
rate with respect to $\omega$. By \eqref{E-defkappa}
\[
\varkappa(L_{\omega}) \le \nu := \tanh\left(\frac{1}{4}\delta\right)<1, \quad \forall \omega\in\Omega.
\]
Now let $\bp^{(n)}(\omega)$ be the $n$th iterate of the random Markov chain
\eqref{DDE}. Then
\[
 \rho_H\left(\bp^{(n+1)}(\omega),a_{\theta_{n+1} \omega}\right) \leq \nu \, \rho_H\left(\bp^{(n)}(\omega),a_{\theta_n \omega}\right)
\]
for all $n \geq 0$ and every $\omega \in \Omega$, since
\[
\bp^{n+1}(\omega) = L_{\theta_n \omega} \bp^{(n)}(\omega),
\quad a_{\theta_{n+1} \omega} = L_{\theta_{n} \omega} a_{\theta_n \omega}.
\]
Hence,
\[
 \rho_H\left(\bp^{(n)}(\omega),a_{\theta_{n} \omega}\right) \leq \nu^n \, \rho_H\left(\bp^{(0)},a_{ \omega}\right)
\]
for all $n \geq 0$ and every $\omega \in \Omega$, from which follows the
pathwise forward convergence with respect to the Hilbert projective metric.
\[
 \rho_H\left(\bp^{(n)}(\omega),a_{\theta_{n} \omega}\right) \to 0 \quad \textrm{as}\quad n \to \infty, \qquad \forall \omega\in\Omega.
\]
By Remark \ref{R-unbound}, convergence in the Hilbert projective metric
implies convergence in any norm on $\mathbb{R}^N$. This gives

\begin{corollary}
For any norm $\|\cdot\|$ on $\mathbb{R}^N$, $\bp^{(0)} \in\simpN$ and
$\omega \in \Omega$
\[
\left\|\bp^{(n)}(\omega) - a_{\theta_{n} \omega}\right\| \to 0 \quad \textrm{as}\quad n \to \infty, \qquad \forall \omega\in\Omega.
\]
\end{corollary}
The random attractor is, in fact, asymptotic Lyapunov stable in the
conventional forward sense.

\subsection{Deterministic nonautonomous Markov chains}\label{S-discus}

The above proofs make no use of probabilistic properties of the sample path
parameter $\omega$ (apart from $\mathcal{F}$-measurability considerations,
which are not an essential part of the proof). It applies immediately to
deterministic nonautonomous Markov chains in which the transition
probabilities vary, say, periodically in time.

As described in \cite{PKMR}, this time variation can be modelled by letting
$\omega$ be an bi-infinite sequence $\omega =(\omega_n)_{n\in \mathbb{Z}}
\in \Lambda^{\mathbb{Z}}$, i.e., with $\omega_n \in \Lambda$, $n \in
\mathbb{Z}$, for some compact metric space $(\Lambda,\rho_{\Lambda})$. Then
$\Omega =\Lambda^{\mathbb{Z}}$ is a compact metric space with the metric
\[
\rho_{\Omega}(\omega,\bar{\omega}) =
\sum_{n\in \mathbb{Z}} 2^{-|n|} \rho_{\Lambda}(\omega_n,\bar{\omega}_n)
\]
and the shift operator $\theta (\omega_n)_{n\in \mathbb{Z}}
=(\omega_{n+1})_{n\in \mathbb{Z}}$ is continuous in the metric
$\rho_{\Omega}$. It turns out then that $\omega \mapsto a_{\omega}$ is
continuous here (in general, the set-valued mapping $\omega \mapsto
A_{\omega}$ is only upper semi-continuous). These topological properties of
the driving system replace the measurability properties in the random
dynamical systems.

  \providecommand{\bbljan}[0]{January} \providecommand{\bblfeb}[0]{February}
  \providecommand{\bblmar}[0]{March} \providecommand{\bblapr}[0]{April}
  \providecommand{\bblmay}[0]{May} \providecommand{\bbljun}[0]{June}
  \providecommand{\bbljul}[0]{July} \providecommand{\bblaug}[0]{August}
  \providecommand{\bblsep}[0]{September} \providecommand{\bbloct}[0]{October}
  \providecommand{\bblnov}[0]{November} \providecommand{\bbldec}[0]{December}

\end{document}